\documentclass[12pt]{article}

\usepackage{amssymb}
\usepackage{amsmath, amscd}
\usepackage{amsthm}
\usepackage{bookman}
\usepackage{eepic}
\newcommand{\di}{\mbox {dist}}
\newcommand{\vy}{\underline{y}}
\newcommand{\vz}{\underline{z}}

\newcommand{\vx}{\underline{x}}
\newcommand{\vd}{\underline{d}}
\newcommand{\va}{\underline{a}}
\newcommand{\vb}{\underline{b}}
\newcommand{\vc}{\underline{c}}
\newcommand{\M}{\mathfrak{M}}

\newtheorem{lemma}{Lemma}[section]
\newtheorem{defn}[lemma]{Definition}
\newtheorem{axiom}{Axiom}
\newtheorem{cor}[lemma]{Corollary}
\newtheorem{remark}[lemma]{Remark}
\newtheorem{prop}[lemma]{Proposition}
\newtheorem{thm}[lemma]{Theorem}

\title{Metric spaces and SDG}

\author{Anders Kock}
\date{}
\begin{document}
\maketitle

\section*{Introduction} The first use of the term ``Synthetic differential geometry'' seems 
%\end{document}
to be in Busemann's \cite{BuJ}, \cite{BuR} (1969 and 1970), so it is prior to the 
use of the term in, say, \cite{SDG}, where the reasoning goes in 
a different direction.  In Busemann's work, the basic 
structure is that of  {\em metric space}. This notion has not been much 
considered in the context of SDG\footnote{We shall use the acronym ``SDG'' when 
referring to the school originating with Lawvere's ``Categorical 
Dynamics'' talk (1967) and the KL Axiom (``Kock-Lawvere''), see \cite{SDG}, \cite{MR}, 
\cite{Lav}, \cite{SGM}, \ldots } (except 
in its infinitesimal form: {\em Riemannian} metric).
One reason for this 
is: to provide the number line $R$ with the standard metric, 
one needs the absolute value function $R\to R$, given by  $x\mapsto 
|x|$. 
This map, however, is not smooth at $0$; and in SDG, only smooth maps can be considered.

However, the geometric reasoning of Busemann has a genuine 
synthetic character.  It does admit co-existence with SDG, which I hope 
that the present note will illustrate.

\section{Metric spaces and the neighbour relation}\label{msnrx}
To have a metric on a set $M$, one needs a number line $R$ to receive 
the values of the metric, usually the ring of real numbers.

Recall that it is essential for SDG that the number line $R$, with its  
ring structure, has a rich supply of nilpotent elements, in 
particular, elements $\epsilon$ with $\epsilon ^{2 }=0$. Such 
nilpotent elements lead to the geometric notion of when two points in 
a manifold are (first order) {\em neighbours}, written $x\sim y$. 
When the manifold itself is the number line $R$, then $x\sim y$ will mean 
$(x-y)^{2}=0$.

How do such nilpotent elements coexist with the metric? 
We hope to demonstrate not only that they do coexist, but they enhance Busemann's metric-based 
differential geometric notions, by allowing a notion for when two subspaces of $M$ {\em touch} each other (tangency), leading to e.g.\ the envelopes and wave fronts, occurring already in Huygens' work. 

We shall 
study this axiomatically, with  intended  application only for the special 
case where the metric space $M$ is 
just a Euclidean space, built on basis of the given number line $R$.
In particular we study ``lines'' (or rays or geodesics) in $M$. Lines occur as a derived concept only. 
 We are not using the full range of  algebraic 
properties of $R$, but only the addition and order properties of the 
positive part $R_{>0}$. 
A model for the axiomtaics are presented in Section \ref{modelx}; it depends on having a model for the axiomatics of SDG. Therefore, it may be that no models exist in the category of (boolean) sets. We are really talking about interpretations and models of the theory  in some topos $\mathcal{E} $ or other suitable category; nevertheless, we shall talk about the objects in  $\mathcal{E}$, as if they were just sets. This is the common practice in SDG.

    \subsection{A basic picture}\label{baspx}

``{\em The shortest path between two points is the straight line.}''

\medskip

This may be seen as a way of describing the concept of ``straight 
line'' in terms of the more primitive concept of {\em distance} 
(which may be measured by ``how long time does it 
take to go from the one point to the other'' -- like ``optical 
distance'' in geometrical optics, cf.\ e.g.\ \cite{Arnold}). 

Consider an obtuse triangle, with height 
$\epsilon$ at the obtuse angle (cf.\ figure below). In coordinates, with $b$ as origo  
$(0,0)$,
 the vertices are $a=(-r,0)$, $b'= 
(0,\epsilon )$, and $c= (s,0)$. 
%\end{document}
The height divides the triangle in two right triangles: the one triangle
has catheti of lengths $r$ and $\epsilon$, and the other one has catheti of 
lengths $\epsilon$ and $s$. If $\epsilon ^{2}=0$, the length  of the 
hypotenuse of the
first triangle is then, by Pythagoras, $\sqrt{r^{2}+\epsilon ^{2}} = 
\sqrt{r^{2}} =r$,  
and the length of the hypotenuse of the other triangle is similarly 
$s$.

So the path from $a$ to $c$ via $b'$ has length $r+s$, just as the 
straight line from $a$ to $c$.

What distinguishes, then, in terms of length, the straight line from 
the path via $b'$? Both have length 
$r+s$, the minimal possible length.

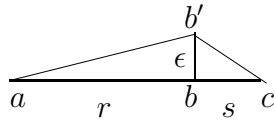
\begin{picture}(80,60)(-120,20)
\put(20,30){\line(1,0){95}}
\put(20,30){\line(4,1){70}}
\put(90,30){\line(0,1){18}}
\put(90,47){\line(3,-2){27}}
\put(20,20){$a$}
\put(86,20){$b$}
\put(53,16){$r$}
\put(86,50){$b'$}
\put(82,35){$\epsilon$}
\put(115,20){$c$}
\put(100,16){$s$}
\end{picture}
\begin{equation}\label{bpx}
{}
\end{equation}

The answer is not in terms of extremals, but in terms of {\em 
stationary} or  {\em critical} values ; the length of the path via $b$ is stationary for 
``infinitesimal variations'' (e.g.\ via $b'$), unlike the path via 
$b'$, in the sense that we shall make precise by the notion of {\em 
focus}: $b$ is the focus of the set of points 
$b'\sim b$ with distance $r$ to $a$ and distance $s$ to $c$.

\medskip

We intend here to give an axiomatic theory, involving a set $M$ equipped with a metric $\di$ (in a certain restricted sense we shall make precise below), and a reflexive symmetric "neighbour" relation $\sim$. The metric is assumed to take values in an unspecified number line $R$ with a total strict order relation $>$;
we will only use a few  properties of the number line $R$, namely the additive and order-properties of $R_{>0}$.

The axiomatics which we present has models, built on basis of (models of) SDG; we relegate the discussion of this until the end of the paper (Section \ref{modelx}) , to stress the fact that the theory we develop is in principle prior to any coordinatization of the geometric material.

   \subsection{Metric spaces}\label{MSx}

A {\em metric space} is a set $M$ equipped with\footnote{one may take "$x\# y$" to mean  "$x\neq y$,
in which case it is not an added structure; however, our reasoning will not involve any negated assertions.}
  an apartness relation $\#$, 
assumed symmetric, and a symmetric function 
$\di : M\times_{\#} M \to R_{>0}$,
i.e.\ 
$\di (a,b)= \di (b,a)$ 
for all $a, b$ in $M$ with $a$ and $b$ 
apart (here, $M\times_{\#} M$ denotes the set of $(a,b)\in M \times M$ 
with $a\# b$). Since 
$\di (a,b)$
 will appear in quite a few formulae, we use 
Busemann's 
short notation:
$$\di (a,b) \mbox {\quad is denoted \quad} ab.$$ 

\medskip

The triangle inequality, $ac\leq ab+bc$ will play no role in the 
present 
note, except when it happens to be an equality, $ac=ab+bc$ (which is a property that a triple of points 
$a,b$, $c$ may or may not have). In fact, we will not be using  
the relation $\leq$  in the present note. 

 \medskip

We follow Busemann in writing $(abc)$ for the statement that
 the triangle equality
holds for three points $a,b$, $c$ (mutually apart); thus
$$(abc) \mbox{\quad  means \quad } ab+bc=ac.$$

Note that $(abc)$ implies $ab<ac$ and $bc<ac$.
Classically, $(abc)$ is expressed verbally: ``the points $a,b,c$ are 
collinear (with $b$ in between $a$ and $c$)". But $(abc)$ will be weaker than collinearity, in our 
context: 
for, $(ab'c)$ holds in the basic picture (\ref{bpx}) above, but $a,b'$ $c$ 
are not collinear. We shall 
below give a stronger notion $[abc]$ of collinearity. 

%We elaborate on 
%the $(abc)$-notation (and later on, similarly on the 
%$[abc]$-notation) by writing $(abc)^{rs}$ if 
 %$ab=r$ , $bc=s$ {\em and}  $ac=r+s$.

\medskip

The sphere 
$S(a,r)$ with center $a$ and  radius $r>0$ is defined by
$$S(a,r):= \{b \in M \mid ab=r \}.$$

%If $H\subseteq M$, we say that a point $a$ is {\em apart} from $H$ if $a\#x$ for every $x\in H$.

   \subsection{The neighbour relation $\sim$}\label{nbrx}
 Some uses of the (first order) neighbour relation $\sim$ were 
described in \cite{SDG} \S I.7, 
and the neighbour relation is the basic notion in \cite{SGM}.
Knowledge of these or related SDG  texts is not needed in the 
following, except for 
%Sections \ref{modelx}, 
where we, in Section \ref{modelx}  construct a 
model of the present axiomatics.
 
Objects $M$ equipped with a reflexive 
symmetric relation $\sim$, we call {\em manifolds}, for the present 
note. Any subset $A$ of $M$ inherits a manifold structure from 
$M$, by restriction. The manifolds we consider in the present note 
are such subsets of a fixed $M$. 

The motivating examples (discussed in Section \ref{modelx} ) are the $n$-dimensional coordinate vector 
spaces $R^{n}$ over $R$, where $\vx \sim \vy$ means $(y_{i}-x_{i})\cdot 
(y_{j}-x_{j})=0$ for all $i,j = 1, \ldots ,n$ (where $\vx = (x_{1}, 
\ldots ,x_{n})$ and similarly for $\vy$). We
 refer the reader to the SDG literature (notably \cite{SGM}) for an 
exploitation of the notion for more general manifolds. The main 
aspect is that any (smooth) map preserves $\sim$. Note that in 
particular $x\sim 0$ in $R$ iff $x^{2}=0$. We shall also in the axiomatic treatment, e.g.\ in the proof of Lemma \ref{mainx}, assume that all maps constructed  preserve $\sim$.

\medskip

If $M$ is furthermore equipped with a metric, as described in the 
previous Subsection, there is a compatibility requirement, namely $x\# y$ 
and $y\sim y'$ implies $x\# y'$; and there is an incompatibility requirement: 
$x\# y$ and $x\sim y$ are incompatible, i.e.\ $\neg ( (x\# y )\wedge (x\sim 
y))$.

\medskip

It is useful to make explicit the way the metric and the neighbour 
relation ``interact'', in the case where $M$ is  the number line $R$ itself,
where we take the distance $xy$ to mean $|y-x|$ (for $x\# y$) and take 
$x\sim y$ to mean $(y-x)^{2 }=0$. Note that the numerical-value 
function used here is  smooth on the set points $x\# 0$. 
For $R^{n}$ and other manifolds as a model, in the 
context of SDG, see Subsection \ref{modelx} below.

Neighbours of $0$ in $R$ are in SDG called first order 
{\em infinitesimals}. They have no influence on the order of $R$; it 
is a standard calculation that

\begin{prop} If $x<y$,  and $\epsilon \sim 0$, then $x+\epsilon <y$.
\end{prop}

\medskip

Note that if we define $x\leq y$ to mean that ``$x$ is not $>y$'', then,
for $\epsilon \sim 0$,we have that  $\epsilon$ is not $>0$, and 
similarly $\epsilon$  is not 
$<0$. So $\epsilon \leq 0$, and also $\epsilon \geq 0$. So the relation  
$\leq$ is only a preorder, not a partial order (unless $\epsilon^{2}=0 
$ implies $\epsilon =0$), and so $\leq$ cannot in general be used to 
determine elements in $R$ uniquely.

\medskip
An example of the relation $\sim$ on a manifold is equality: $x\sim 
y$ iff $x=y$. If this is the case, we say that $\sim$ is {\em 
trivial} or that $M$ is {\em discrete}. So the theory we are to develop for manifolds with 
metric have as a special case (a fragment of)
Busemann's theory. The reason  for introducing the $\sim $ 
relation is that it allows one to express, in geometric terms and 
without explicit  
differential calculus,  the 
notion of a {\em stationary} (or {\em critical}) value of a function defined on $M$. 
(This notion is of course related to the notion of {\em extremal} value of a 
function; for the present purposes, extremal value is not 
so relevant as stationary value.) 

\medskip

We recall some notions derived from a neighbour relation $\sim$ on 
$M$ (see also  \cite{SDG} I.6 and \cite{END}).
\medskip

For $z\in M$, we denote by $\M (z)$ the 
set of $z'\in M$ with $z'\sim z$, and we call it the (first order) ``monad'' around 
$z$.
A function (typically "distance from a given point $ z$") $\delta : M\to X$, defined on $M$, is said to have $z\in M$ as a {\em stationary value} if $\delta$ is constant on $\M (z)$.

\medskip

\begin{defn} Let $A$ and $B$ are subsets of $M$, and $z\in A \cap B$. We say that  $A$ {\em 
touches} $B$ at $z$, (or that $A$ and $B$ have  {\em at least first order 
contact} at $z$) if  
$$\M (z) \cap A = \M (z) \cap B.$$
 \end{defn}
Equivalently:
for all $z'\sim z$ in $M$, we 
have $z'\in A$ iff $z'\in B$.
 ``Touching at $z$'' is clearly an 
equivalence relation on the set of subsets of $M$ that contain $z$.

Note that if $A$ touches $B$ in $z$, we have $\M (z)\cap A \subseteq A\cap B$ and $\M (z)\cap B 
\subseteq A\cap B$.

\begin{defn} A subset $N\subseteq M$ will be called {\em focused} if there is a 
unique $n\in N$ so that $n'\sim n$ for all $n'\in N$. This unique $n$ 
may be called the {\em focus} of $N$.
\end{defn} 
Clearly, any singleton set is focused.
If $\sim$ is trivial (or more generally, if $\sim$ is transitive), then singleton subsets  
are the only focused subsets.
Note that $N$ being focused is a property of $N$, and the focus of 
$N$  is 
not an added structure.

Two subsets $A$ and $B$ of $M$ may touch each other in more than one point $z$. 
We are interested in the case where they touch each other in {\em exactly} one point $z$, and $\M (z)\cap A $ ($= \M (z) \cap B$)  is focused (then necessarily with $z$ as focus). We then say that $A$ and $B$ have {\em focused} touching; in this case, we call $z$  {\em the  touching point} (note the definite article), and we call $\M(z)\cap A = \M (z)\cap  B$ {\em the touching set} of $A$ and $B$. (It may be strictly smaller than $A\cap B$, see Remark \ref{redherring} below.)

In the intended application in SDG, we have, for spheres in $R^{n}$, with the standard Euclidean metric,
the following facts, which we here take as an axioms:
\begin{axiom} \label{dimx} Given spheres $A$ and $C$ in $M$, and given  $b\in A\cap 
C$. Then: $$\M (b) \cap A \subseteq \M (b)\cap C\mbox{\quad   
implies  \quad  }\M (b) \cap A = 
\M (b)\cap C.$$
\end{axiom}
This is essentially because the spheres have the same dimension. A proof of validity of this axiom in the context of
 the standard SDG axiomatics is given in Proposition \ref{fu3x} below.

\begin{axiom} \label{clex}If two  spheres $A$ and $C$  in $M$ (whose centers are apart) touch each other, then the touching is focused. \end{axiom}
 The validity of this Axiom in the intended model for the axiomatics 
is  argued in Section  \ref{modelx} below (Proposition \ref{AHx}).

\medskip
The following gives a characterization of the focus asserted in Axiom 
\ref{clex}. Let $A$ and $C$ be as in the Axiom.
\begin{prop}\label{chrrx}Assume $b\in A\cap C$, and assume that for all $b'$, we have
\begin{equation}\label{chrrxx} (b'\sim b \wedge b'\in A) \Rightarrow  
b'\in C.\end{equation} Then $b$ is the touching point of $A$ and $C$.
\end{prop}
\begin{proof} The assumption (\ref{chrrxx}) gives $\M (b)\cap A 
\subseteq \M(b) \cap C$, and then Axiom \ref{dimx} gives $\M (b)\cap A 
= \M(b) \cap C$. So $A$ and $C$ touch at $b$. \end{proof}

\section{Touching of spheres}
Let $M$ be any metric 
space, with a neighbourhood relation $\sim$, as in  
\ref{MSx} and \ref{nbrx}.
For two spheres in $M$, one has two kinds of touching, external and internal.
External touching occurs when the distance between the 
centers equals the sum of the radii, and internal touching when the distance between the centers 
is the (positive) difference between the radii. In elementary Euclidean geometry, the 
differential-geometric concept of ``touching'' may, for spheres, be 
replaced by the more primitive concept of ``having precisely one 
point in common'', so classical synthetic geometry circumvents 
bringing in differential calculus for descri\-bing the touching of two {\em spheres}. 
In our context, the differential calculus is replaced by use the 
 notion of touching derived from the synthetic neighbour relation 
$\sim$, as described in Section \ref{nbrx}; it is applicable to {\em any} two subspaces of $M$.  The classical criteria for 
touching in terms of the distance between the centers of spheres then look the 
same as the classical ones, except that the  meaning of the word 
``touching'' is now the one defined using $\sim$. These criteria we take as axioms:

\begin{axiom} \label{externalx}[External touching] Let $A=S(a,r)$ and let  
$C=S(c,s)$  with $ac>r$.  Then the following conditions are equivalent:

1) $A$ and $C$ touch each other

2) $ac=r+s$.

\end{axiom}
The touching point of $A$ and $ C$ implied by  1) and Axiom \ref{clex} is 
denoted $b$ in the following picture. We shall use the notation 
$a\triangleleft_{s}c$ for $b$; $r$ need not be mentioned explicitly, it is $ac-s$.

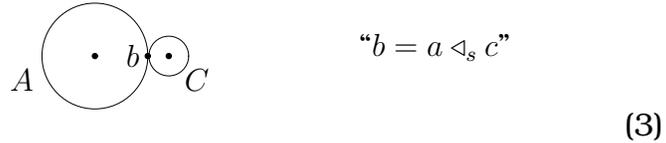
\begin{equation}\label{bPictx}
\begin{picture}(100,50)(0,0)
\put(-2,17){$A$}
\put(30,30){\circle{40}}
\put(30,30){\circle*{2}}

\put(64,17){$C$}
\put(58,30){\circle{15}}
\put(58,30){\circle*{2}}
%\put(30,17){$B$}
\put(50,30){\circle*{2}}
\put(42,26){$b$}
\put(130,30){``$b= a\triangleleft _{s}c$'' }
\end{picture}
\end{equation}

\noindent Using Proposition \ref{chrrx}, this $b$ may be characterized by 
\begin{equation}\label{charb1x}\mbox{ for all }b'\sim b:ab'=ab \Rightarrow b'c= bc
\end{equation} 
and also by
\begin{equation}\label{charb2x}
\mbox{ for all }b'\sim b:b'c =bc \Rightarrow ab'= ab.
\end{equation} 

Since $b\in A\cap C$, we have $ab=r$ and $bc=s$, and since we also 
have $ab+bc=ac$ (by 2) in the Axiom), we have the triangle equality 
$ab+bc=ac$; recall the notation   $(abc)$ for this equality.

\begin{axiom} \label{internalx}[Internal touching] Let  $A=S(a,r+s)$ and 
$B=S(b,s)$ with $ab<r+s$. Then the following  conditions are 
equivalent:

1) $A$ and $B$ touch each other 

2) $ab=r$.
\end{axiom}
(Note that  the sphere $A$ here is not the same as $A$ in 
the previous Axiom; it is bigger.)

The touching point  of $A $, $ B$ implied by 1) and Axiom \ref{clex} 
is denoted $c$ in the following picture. We shall use the notation 
$a\triangleright_{s}b$ for $c$; again $r$ need not be mentioned explicitly.

\medskip

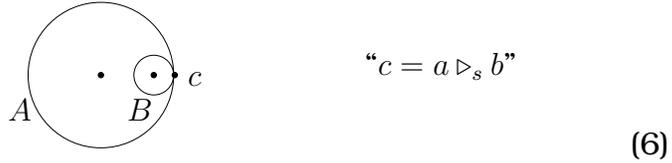
\begin{equation}\label{cPictx}
\begin{picture}(100,50)(0,0)
\put(-5,13){$A$}
\put(40,13){$B$}
%\put(30,30){\circle{40}}
\put(30,30){\circle*{2}}
\put(30,30){\circle{55}}
\put(50,30){\circle{15}}

%\put(64,17){$C$}
%\put(58,30){\circle{15}}
\put(58,30){\circle*{2}}
\put(63,26){$c$}
%\put(30,17){$B$}
\put(50,30){\circle*{2}}
%\put(42,26){$b$}
%\put(130,30){``$b= a\triangleleft _{s}c$'' }
\put(130,30){``$c= a\triangleright _{s}b$'' }
\end{picture}
\end{equation}

\noindent Using Proposition \ref{chrrx}, this $c$ may be characterized by 
\begin{equation}\label{charc1x}\mbox{ for all }c'\sim c:ac'=ac 
\Rightarrow bc'= bc
\end{equation} 
and also by
\begin{equation}\label{charc2x}
\mbox{ for all }c'\sim c:bc' =bc \Rightarrow ac'= ac.
\end{equation}
Again,  we have the triangle equality 
$ab+bc=ac$.  

\subsection{Interpolation and extrapolation}
We shall describe how the Axioms for external and internal  touching of spheres 
give rise to an interpolation process and to an extrapolation 
process, respectively.

More precisely, given two points $a$ and $c$ (with $a\# c$) 
and given  a number $s$ with $0<s<ac$.  
 Consider the two spheres
$$S(a, ac-s) \mbox{ and } S(c,s).$$
The sum of the two radii is $ac$, so  the Axiom for external touching 
states that the touching of $S(a, ac-s) $, $ S(c,s)$ is focused. Denote the touching point  by
$a\triangleleft_{s} c$. This is the $b$ depicted in (\ref{bPictx}), 
with $r=ab$, $s=bc$.
Note that $ab+bc= ac$, or in Busemann's notation $(abc)$.

Also, given two points $a$ and $b$, and given an arbitrary  
number $s>0$. Consider the two spheres
$$S(a,ab+s) \mbox{ and } S(b,s).$$
The difference of the two radii is $ab$, so the Axiom for internal 
touching states that the touching of $S(a,ab+s)$, $ S(b,s)$ is focused. Denote the touching point by $a\triangleright _{s}b$. This is the $c$ depicted in 
(\ref{cPictx}) with $r=ab$, $s=bc$.
Note that we also here have $(abc)$.
 
\medskip
The notation $a\triangleleft_{s}c$ suggests that $a\triangleleft_{s}c$ 
is the point 
obtained by moving $s$ units from $c$ in the direction from $c$ to $a$; it is 
an interpolation, since $b$ is in between $a$ and $c$, by $(abc)$.
Likewise  $a\triangleright _{s}b$ is the point obtained 
by moving $s$ units away  from $b$ in the direction given by the 
``vector'' from $a$ to $b$; it is an extrapolation. 

\medskip

The possibility of extrapolation is a basic axiom in Busemann's 
synthetic geometry, Axiom D in \cite{2points}, Section II. Geometrically, this axiom 
says that any line segment from a point $a$ to another point $b$ 
may be  extrapolated (prolonged) beyond $b$ by the amount of 
$s$ units say, for {\em certain} $s\in R_{>0}$. The theory we present 
makes a more rude statement about extrapolation, namely that extrapolation for {\em any} 
positive amount $s$ is possible, and this implies that the spaces we consider are 
unbounded. (Busemann was also interested in bounded models for his 
axiomatics, namely e.g.\ elliptic spaces.)

\medskip

 We note that we have constructed a map $b\mapsto 
a\triangleright _{s}b$ from $S(a,r)$ to $S(a, r+s)$, (one should 
think of it as radial projection for two concentric spheres); as any map that can be 
constructed, it preserves $\sim$: if $b_{1}\sim b_{2}$ in $S(a,r)$, 
then  $a\triangleright _{s}b_{1} \sim a\triangleright _{s}b_{2}$ in 
$S(a,r+s)$.

\section{Collinearity}

The triangle equality $ab+bc=ac$, or $(abc)$, for three points  $a,b,c$, is central in
the synthetic differential geometry of Busemann, for defining 
geodesics, and in particular lines. It expresses classically a collinearity property of 
$a,b, c$. In the present version of SDG, based on the neighbour 
relation, $(abc)$ is weaker than collinearity; referring to the ``basic 
picture'' (\ref{bpx}), we do have $(ab'c)$, but $a,b',c$ will not be collinear 
in the stronger sense to be presented; but (again referring to the 
basic picture), $a,b, c$ will. 

The equivalent conditions of the following Lemma will serve as 
defintion (Definition \ref{collinx} below) of when three points $a,b,c$ 
satisfying the triangle 
equality $(abc)$ deserve the name of being collinear in our stronger 
sense:   
\begin{lemma}\label{mainx} Given three points $a,b,c$, mutually 
apart, and satisfying $(abc)$, i.e.\ satisfying the
triangle equality
$ab+bc=ac$.
Then the following six assertions are equivalent:

$$\begin{array}{llll}
\mbox{\em a1}:&&\mbox{for all }a'\sim a:&a'b=ab \Rightarrow a'c= ac\\
\mbox{\em a2}:&&\mbox{for all } a'\sim a:&  a'c= ac\Rightarrow a'b=ab\\
\mbox{\em b1}:&&\mbox{for all }b'\sim b:&ab'=ab \Rightarrow b'c= bc\\
\mbox{\em b2}:&&\mbox{for all }b'\sim b:&b'c=bc \Rightarrow ab'= ab\\
\mbox{\em c1}:&&\mbox{for all }c'\sim c:&ac'=ac\Rightarrow bc'= bc\\
\mbox{\em c2}:&&\mbox{for all }c'\sim c:&bc'=bc \Rightarrow ac'= ac\\
\end{array}$$
\end{lemma}
\begin{proof}We note that b1 and b2 are equivalent: they both express 
that $b$ is the touching point of $S(a,r)$, $ S(c,s)$ (where $r=ab$ and 
$s=bc$), as we observed  
in (\ref{charb1x}) and (\ref{charb2x}), i.e.\ they express $b=a\triangleleft_{s}c$.
 Similarly c1 and c2 are equivalent: they both express that $c$ is 
 the touching point $c$ in
(\ref{charc1x}) or (\ref{charc2x}), i.e.\  $c=a\triangleright_{t}b$.  Finally, a1  and a2 are equivalent, 
using a change of notation and the equivalence of c1 and c2.

We use b1 to prove c1. Given $c'\sim c$ with $ac'=ac$ ($ =r+s$). Let $b'$ be 
$a\triangleleft_{s}c'$, so $ab'=r$. And $b'\sim b$,  since $a\triangleleft_{s}$ 
preserves $\sim$; furthermore, $b'$ is characterized by
$$\mbox{b'1: for all } b''\sim b' \mbox{  we have }ab'' =r \mbox{ implies }b''c'=s.$$

 \noindent Then since $b\sim b'$ and $ab=r$, we  use b'1 with $b''=b$ to 
conclude $bc'=s$.

Similarly, we use c1 to prove b1: Given $b'\sim b$ with $ab'=r$. Let 
$c'$ be $a\triangleright _{s}b'$, so $ac'=r+s$. And $c'\sim c$ since $a\triangleright _{s}$ 
preserves $\sim$; furthermore  $c' $ is characterized by
$$\mbox{c'1:  for all }c''\sim c'\mbox{ we have }ac'' = r+s \mbox{ implies }b'c''=s .$$

%3b':  for all $c''\sim c'$ we have $b'c''=s $ implies $ac'' = r+s$.

\noindent Then since $c\sim c'$ and $ac=r+s$,  we use c'1  with $c''=c$ to conclude $b'c=s$.

The remaining implications are proved by the same method.\end{proof}

\begin{defn} \label{collinx} Given three points $a$, $b$, and $c$, mutually apart. 
Then we say that $a,b,c$ are {\em collinear (with $b$  in between $a$ 
and $c$)},  and we write $[abc]$, if $(abc)$ holds, and 
one of the six equivalent conditions of  Lemma \ref{mainx} holds. 
 \end{defn}
Note that for $bc=s$, the assertion $[abc]$ is equivalent to $b= a\triangleleft_{s}c$, 
and also 
to $c=a\triangleright _{s}b$. Thus
\begin{prop} \label{abcx}\label{reciprocityx}Given $a,b$, and $s$. Then point $c=a\triangleright _{s}b$ is characterized by
$bc=s$ and the collinearity condition $[abc]$. Also, given $a,c$, and 
$s$ with $s<ac$; then the point 
$b=a\triangleleft _{s} c$ is characterized by the same two conditions. 
In particular, 
$b= a \triangleleft _{s} 
(a\triangleright _{s}b)$, 
and, for $s<ac$, $c=a\triangleright _{s}(a \triangleleft _{s} c)$.
\end{prop}

Geometrically, the last assertion in the  Proposition just describes the   
 bijection between the 
two concentric circles $S(a,r)$ and $S(a,r+s)$ which one obtains by radial 
projection from their common center $a$.

%Similarly, $[abc]$ is equivalent to  $c=a\triangleright_{s}b$.
Sometimes, we shall write $[abc]_{1}$ 
to mean that $[abc]$ holds by virtue of 
a1 or a2, and $[abc]_{2}$ if it holds by virtue of b1 or b2, and 
$[abc]_{3}$ if it holds by virtue of c1 or c2, respectively.  
We clearly have
\begin{equation}\label{1xx} [abc]_{1}\mbox{ \quad iff \quad } 
S(b,ab)\mbox{ touches } S(c,ac)\mbox{ in }a,
\end{equation}
\begin{equation}\label{2xx} [abc]_{2}\mbox{ \quad iff \quad } 
S(a,ab)\mbox{ touches } S(c,bc)\mbox{ in }b,
\end{equation}
\begin{equation}\label{3xx} [abc]_{3}\mbox{ \quad iff \quad } 
S(a,ac)\mbox{ touches } S(b, bc)\mbox{ in }c.
\end{equation}

\noindent Collinearity is  ``associative'', in the following sense. 
Given a list of four points $a,b,c,d$, mutually apart. Consider the 
following  four 
collinearity assertions: $$[abc], [abd], [acd], [bcd].$$

\begin{prop} \label{transx}If two of these collinearity assertions hold, then they all 
four do.
\end{prop}
\begin{proof}  The proofs of the various cases are similar, so we give 
just one of them: we  prove that $[abc]$ and $[acd]$ imply $[bcd]$. The 
point $c$ occurs in all three of these assertions, and we concentrate on that 
point: we  use $[abc]_{3}$ and $[acd]_{2}$ to prove $[bcd]_{2}$. So 
assume that  $c'\sim c$ with $bc'=bc$. By $[abc]_{3}$, we  therefore 
have   
$ac'=ac$. By $[acd]_{2}$ we therefore have the desired $c'd=cd$. This proves 
$[bcd]_{2}$.\end{proof}

Note that since $ab=ba$ etc., the assertion $[cba]_{1}$ is 
the same as  $[abc]_{3}$. From  this
we conclude that
collinearity is 
symmetric: $[abc]$ iff $[cba]$. 
We  say that three points $a,b$, and $c$ are {\em aligned} if some permutation 
of them are collinear (so for the term ``alignment'', we ignore which point is in the 
middle).

\begin{prop}\label{37x}Assume that $[a'ab]$. Then $a'\triangleright _{s}b= 
a\triangleright _{s}b$.
\end{prop}
\begin{proof}  By construction, $a'\triangleright _{s}b$ is aligned 
with $a'$ and $b$,  and $a',a, b$ are aligned by assumption. So we 
have two of the four possible alignment assertions 
for $a',a,b$, and $a'\triangleright _{s}b$.  
From associativity of collinearity 
(Proposition \ref{transx}) we conclude that $a,b$, and $a'\triangleright 
_{s}b$ are aligned (with $b$ in the middle); and $a'\triangleright 
_{s}b$    
has distance $s$ to $b$. These two properties characterize 
$a\triangleright _{s}b$ by Proposition \ref{reciprocityx}.\end{proof}

 The characterization of $a\triangleright_{s}b$ (respectively of 
$a\triangleleft _{s} c$)  
implies 
\begin{prop} \label{22sphx}If two spheres touch another,  then their 
centers are aligned with their touching point.\end{prop}

\subsection{Stiffness} A major aim in Busemann's work, and also in the present note,  is to construct a notion 
of {\em  line} in terms of distance. Basic here is the notion of when 
three $a,b,c$ points are {\em collinear}; classically, this is 
the statement $(abc)$, i.e.\ $ab+bc=ac$; this is, in practical terms, 
to define lines  in terms of taut strings\footnote{the word ``line'' in geometry is derived from 
 ``line'' (thread  made of linen) in textiles.}. This gives you 
something which is only rigid ``longitudinally'', but not 
``transversally'', whereas our stronger notion, defined using $[abc]$, 
further involves transversal rigidity: if $[abc]$ holds, then the 
infinitesimal  
transversal variation given by replacing $b$ by $b'$, as in the basic 
picture (\ref{bpx}), still satisfies $(ab'c)$, 
(whereas   $[ab'c]$ fails).

In practical terms, $(abc)$ refers to lines given by a taut string, 
whereas $[abc]$ refers to lines given by  a {\em ruler} (or {\em 
straightedge}). The 
transversal rigidity, usually called its {\em stiffness}, of a ruler, is achieved by the {\em width} of the 
ruler; the stiffness makes the ruler better adapted than strings for 
{\em drawing} lines, when producing technical drawings on paper.

The stiffness of $[abc]$ is obtained by a a qualitative 
(infinitesimal) kind of 
width, given by the neighbour relation.

\section{Huygens' Theorem for spheres}\label{HTX}

Let $T$ be a manifold, and let $S_{t}$, for $t\in T$, be a family of 
submanifolds of a manifold $M$.

{\em An
envelope}  (note the indefinite article ``an'' ) for the family 
$S_{t}$ ($t\in T$)  is {\em a} manifold $E\subseteq M$ 
 such that every $S_{t}$ touches $E$, and every point in $E$ is 
touched by a unique $S_{t}$. 
(Here, we  used the impredicative, or implicit, definition of 
the notion of envelope.
See \cite{END} for a comparison with more explicit definition, 
equivalent to the ``discriminant'' method, which provides {\em the} 
(maximal) envelope, as the union of the "characteristics").

\begin{thm}\label{Huyx} [Huygens] An envelope $E$ of the $S(b,s)$, as $b$ ranges 
over $S(a,r)$, is  $S(a,r+s)$. For $b\in S(a,r)$, $E$ touches 
$S(b,s)$ in $a\triangleright _{s}b$.
\end{thm}
\begin{proof} 
For $b\in S(a,r)$, $S(b,s)$ touches $S(a,r+s)$ in $a\triangleright 
_{s}b$. Conversely,
let $c\in S(a,r+s)$;  we 
take  $b:= a\triangleleft_{s}c$. The point $b$ is then in $ S(a,r)$, by 
construction. So $S(b,s)$ touches $S(a,r+s)$ in $a\triangleright 
_{s}b$, but since $b=a\triangleleft_{s}c$, this is $a \triangleright 
_{s}(a\triangleleft _{s}c)$, which is $c$, by Proposition 
\ref{reciprocityx}.

The uniqueness of $b$ follows from the fact that radial projection is 
a bijection $S(a,r)\to S(a,r+s)$.\end{proof}

\section{The ray given by two points}\label{ray2x}
The notion of ray to be given now is  closely related to
 what \cite{2points} calls a geodesic, except that a geodesic in $M$ is 
(represented by)  a map $R \to M$, whereas a ray is a map $R_{>0}\to 
M$, so is 
only a ``half geodesic'';  and, furthermore, a ray has, unlike a geodesic, a 
definite 
source or starting point.

\begin{prop} \label{asssx}Let $a$ and $b$ in $M$, with $ab=r$, say. Then for any $s,t \in R_{>0}$, we 
have
$$a\triangleright _{t} (a\triangleright _{s}b)= a\triangleright 
_{s+t}b.$$
\end{prop}
\begin{proof}  Let for brevity $c:= a\triangleright _{s}b$ and $d:= 
a\triangleright _{t}c$. Then $[abc]$ and $[acd]$, hence by Proposition \ref{transx}, we also have 
$[abd]$ and $[bcd]$, Also, by construction, $bc=s$ and $cd=t$. By $[bcd]$ we have $bd=s+t$, and by $[abd]$, $d$ is aligned with $a,b$. These two properties characterize  $a\triangleright_{s+t}b$.
\end{proof}

%\medskip

We call the map $R_{>0}\to M$ given by $s\mapsto a\triangleright 
_{s}b$ the  {\em ray generated by $a$ and $b$}, and we call $b$ its 
{\em source } of the ray. (Note that we 
cannot say $a\triangleright_{0}b =b$, 
since $a\triangleright _{s}b$ only is defined for $s>0$. However, it is easy to "patch" rays, using Proposition \ref{transx}.)

\begin{prop} Any ray $R_{>0}\to M$ is  an isometry i.e.\ is distance preserving. 
Furthermore, any triple of mutually apart points on a ray are 
aligned.
\end{prop}
\begin{proof} The first assertion is an immediate consequence of 
Proposition \ref{asssx}; the second follows from the collinearity of 
$a,b$, and  $a\triangleright_{s}b$ by  iterated use of use of Proposition 
\ref{transx}.\end{proof}
%\ref{asssx}. 

Thus, a ray with source $b$ 
 can be viewed as a 
``parametrization of its image by arc length, measured from $b$'' 
(except that we have not attempted to define these terms here).  
Note that $b$ itself is not in the image of the ray.

The following Example refers to the model of the 
axiomatics which one obtains from SDG, as in Section \ref{modelx}. It shows that the isometry property 
is not sufficient for being a ray:

\medskip

\noindent {\bf Example.}
Consider in $R^{2}$ the points $a=(-1,0), b=(0,0)$, and consider the ray $s\mapsto 
a\triangleright _{s}b$; it is, of course, the positive $x$-axis, 
i.e.\ the map $s\mapsto (s,0)$. But for any $\epsilon$ with 
$\epsilon^{2}=0$,  the map   given by $s\mapsto (s,\epsilon\cdot 
s^{2})$ has the isometry property expressed by $(xyz)$ for any three 
values corresponding to 
$s_{1} <s_{2}<s_{3}$;  
 but the map 
is not a ray, since we 
cannot conclude $[xyz]$ unless $\epsilon 
=0$.

\section{Contact elements}

The notion of contact element is trivial if $\sim$ is discrete, 
for then contact elements are just one-point sets. A one 
point set does not generate a ``ray orthogonal to it'', as the 
contact elements, which  we are to consider, do.
 
\begin{defn} A {\em contact element at $b\in M$} is a  subset $P$ 
of  $M$ which may be written in
the form $\M (b)\cap A$, for some sphere $A$ with  $b\in A$.
\end{defn}
Let $A,b$ and $P$ be as in the definition. We then say that $A$ {\em 
touches $P$ at $a$}. If $C$ is a sphere touching $A$ at 
$b$, we have $$\M (b) \cap C= \M (b)\cap A = P.$$
 Note that $P$ is 
focused set, with focus $a$.

For any sphere $A$ and any $b\in A$, there exists (many) 
spheres $C$ touching $A$ at $b$. (This follows by applying  extrpolation and interpolation). We therefore may equivalently 
describe a contact element at $b$ as the touching set of two spheres, touching another at $b$.

In Subsection \ref{rayQx}, we will  refine the notion into that of a {\em 
transversally oriented} contact element.

\medskip

In the intended applications, the contact elements in $M$ make up the 
total space of the
projectivized cotangent bundle of $M$.

   \begin{defn} \label{Orthogx}
Given a contact element $P$ at $b\in M$, and given $c\# b$. We 
say that {\em $c$ is orthogonal to $P$}, written $c\perp P$, if for all $b'\in P$, $b'c=bc$.
\end{defn}
(Thus, if $\sim$ is trivial, then {\em all} points $c\# b$ are orthogonal to 
$P$.) We clearly have
$c\perp P$ iff $P\subseteq 
S(c,s)$ (where $s$ denotes $bc$).

\begin{prop} \label{perpxx} Assume that $[abc]$ holds. Let $P$ be a contact 
element at $a$. Then
$b\perp P$ implies $c\perp P$ (and vice versa).\end{prop}
\begin{proof}  Assume $b\perp P$. Let $r$ denote $ab$ and $s$ denote $bc$, so $ac=r+s$. For 
any $a'\sim a$, $a'b=r$ implies $a'c=r+s$, by the 
$[abc]$-assumption (in the manifestation a1 in Lemma \ref{mainx}). Also $P \subseteq \M (a)$. Since $a'b=r$ for all 
$a'\in P$, we therefore 
have $a'c = r+s$ for all $a'\in P$, which is the condition $c\perp P$. The other 
implication is similar. \end{proof}

Similarly, if $[abc]$ and if $P$ is a contact element at $b$, 
we have that $a\perp P$ iff $c\perp P$. Finally, if $[abc]$  and if $P$ is a contact element at 
$c$, we have that $a\perp P$ iff $b \perp P$.

\subsection{Transversal orientation of contact elements}\label{rayQx}

Given a contact element $P$. The set of spheres touching 
$P$ falls in two classes: two such spheres are in the same
class if they touch another internally. To provide a contact element 
$P$ with a 
{\em transversal orientation} means to select one of these two classes of 
spheres; the selected spheres we describe as those that touches $P$ on 
the {\em negative} side (we also say: on the {\em inside}).

\medskip

Just as a contact element at $c$ may be presented as the touching set of any two 
spheres which touch each other at $c$, a transversally oriented 
contact element at $c$ may be presented as the touching set of any two spheres 
touching another, from the inside,  at $c$.

\medskip

Let $P$ be a contact element at $b$, and let $c\perp P$ and $bc=s$. 
Then the sphere $S(c,s)$ 
touches $P$. Let $P$ be 
equipped with a transversal orientation; 
then we say that $c$ 
{\em is on the positive side} of $P$ if the sphere $S(c,s)$ touches 
$P$ from the outside.

\medskip

\subsection{The ray given by a contact element} \label{85x} Recall that  $c=a\triangleright _{s}b$ is 
characterized by $bc=s$ and $[abc]$.
This gives rise to another characterization of $c=a\triangleright 
_{s}b$ in terms of $\perp$: let $P$ be the transversally oriented contact element $\M(b) 
\cap S(a,r)$,  where $S(a,r)$ touches $P$ on the inside. Then:  
if $c\perp P$ with $c$ on the positive side of $P$,  and  $bc=s$, then $c= 
a\triangleright _{s}b$. This follows from Proposition \ref{22sphx} (with $s=bc$).

Given a transversally  oriented contact element $P$ at $b$, and given an 
$s>0$. We shall describe a point $P\vdash s$ by the following 
procedure: pick a sphere $A=S(a,r)$ touching $P$ from the inside (so $r=ab$), so $[abc]$ with $c = a\triangleright _{s}b$. It 
follows from the above   that 
this only depends on the $s$ and the transversally oriented contact element $P$, 
but not on any particular choice of the sphere $A$ touching $P$ from 
the inside; we 
put $P\vdash s := a\triangleright _{s}b$. 
The notation suggests graphically the fact  that this point is 
orthogonal to $P$, at distance $s$.   The {\em ray generated by $P$} is defined by $s\mapsto P\vdash s$.

  \subsection{Inflation of spheres} \label{inflx}  
 Given a sphere $S(a,q)$, and given $t>0$. The {\em $t$-inflation} (or the {\em $t$-dilatation}) of 
this sphere is by definition the sphere $S(a,q+t)$.
\begin{prop} \label{consx}
If two spheres touch each other internally, the two $t$-inflated spheres likewise touch each other internally.
If the two first spheres are $S(a,r+s)$ and $S(b,s)$, respectively, with touching point $c$, then the touching point of the two $t$-inflated spheres $S(a,r+s+t)$, 
$S(b,s+t)$ is 
$a\triangleright _{s+t}b = a\triangleright _{t}(a \triangleright _{s}b) = a\triangleright _{t}c = b \triangleright_{t}c$.
\end{prop}
\begin{proof} 
The proof of the touching assertion is identical to the classical proof, using that internal touching of spheres is equivalent to: difference of the radii  equals distance between centers; for our notion of touching, this is Axiom  \ref{internalx}. For the second assertion:
By definition of the $\triangleright$-construction, the first expression here is the touching point of the inflated spheres; it equals the next expression by Proposition \ref{asssx}. It in turn equals the third expression, since $c=a \triangleright _{s}b$ by construction. Finally, the fourth expression follows by Proposition \ref{37x}  from collinearity of $a,b$ and $c$. \end{proof}

There is a similar result for external touching, but then one of the 
centers has to be moved further away.

   \subsection{Flow of contact elements}\label{FCEx}
Given a transversally oriented contact element $P$ at $b$, as in Subsection 
\ref{rayQx}. The construction of the ray $s\mapsto P\vdash s$ 
given there can be enhanced to a parametrized family of transversally 
oriented contact elements $s\mapsto P\Vdash s$, with $P\vdash s$ 
as focus of $P \Vdash s$. Pick, as in Subsection \ref{rayQx}, a 
sphere $A = S(a,r)$ touching $P$ from the inside at $b$. The inflated 
sphere $S(a,r+s)$ contains $a\triangleright _{s}b = P \vdash s$, 
so we get a contact element $\M (P\vdash s)\cap S(a,r+s)$, and we 
take this as $P\Vdash s$. We have to see that this is independent of the 
choice of $A$. We know already that the focus $P\vdash s$ is 
independent of the choice. If we had chosen another $A'$ to represent 
$P$, the spheres $A'$ and $A$ touch each other from the inside at 
$b$, hence their $s$-inflated versions likewise touch each other from 
the inside, at $P\vdash s$, by Proposition \ref{consx}. This means 
that they  define the same 
contact element at this point.

Another description of this contact element $P\Vdash s$, again only see\-mingly 
dependent on the choice of the sphere $A=S(a,r)$, is 
$$P\Vdash s := \M (a\triangleright_{s}b)\cap S(a,r+s).$$

\section{Huygens' Theorem for hypersurfaces}\label{htHsx}
\begin{defn}  A 
{\em hypersurface} in $M$ is a subset $B\subseteq M$ which satisfies: 
for every $b\in M$, $\M (b) \cap B$ 
is a contact element. 
 
To give such $B$ a {\em transversal orientation} is 
to give every such contact element a transversal orientation.
\end{defn}

For suitable $s>0$, we aim at describing ``the {\em parallel surface} to 
$B$ at distance $s$ (in the positive direction)''.

Consider a transversally oriented hypersurface $B$. For each $b\in 
B$, we have a transversally oriented contact element $B(b):=\M (b) \cap 
B$, and therefore we have the ray which it generates.
 
If a point $x\in M$ has $x\perp B(b)$, one says that $b$ is a {\em 
foot} of $x$ on $B$. A given $x$ may have several feet on $B$; thus if 
for instance 
$x$ is the center of a sphere $B$, then every point $b\in B$ is a 
foot of $x$ on $B$.

Now consider, for a given $b\in B$, the ray generated by the 
transversally oriented contact element $B(b)$.  
Every point $x$ on this ray has 
$b$ as a foot on $B$.  We assume that for sufficiently small $s$, 
$b$ is the unique foot of $B(b)\vdash s$ on $B$. 
For given such $s$, 
we denote the set of points obtained as $B(b)\vdash s$ for some 
$b\in B$, by $B\vdash s$. Thus we have a bijection $B \to (B \vdash s)$.
Denote $B \vdash s$ by $C$, so by assumption, there is a bijection 
between $B$ and $C$, with $b\in B$ and $c\in C$ corresponding under 
the bijection if $b$ is the foot of $c$ (equivalently, if 
$c=B(b)\vdash s$).

We have to make the following assumption, which in the intended 
application is a weak one: if a point $c$ has a unique foot on $B$, then so 
does any $x\sim c$, and the two feet are neighbours.
\begin{prop} Under this assumption: if $b\in B$ and $c\in C$ correspond, then $\M(c) \cap 
S(b,s) = \M(c) \cap C$. In particular, $C$ is a hypersurface.
\end{prop}
\begin{proof} Let $x\in \M (c)\cap S(b,s)$.  Let $b'$ be the foot of $x$ on $B$. Since $x\sim c$, 
$b'\sim b$. Since $bx=s$, and $x\perp B(b')$, we therefore have 
$b'x=s$. But $x\perp B(b')$ and $b'x=s$ characterizes the point on $C$ 
corresponding to $b'$; so $x\in C$.

Conversely, let  $c'\in C$ and $ c'\sim c$. Then $c'$ corresponds to a 
point $b'\in B$ with $b'\sim b$, implying that $c'=B(b')\vdash s$, 
hence $b'c'=s$; since $c'\perp B(b')$ and $b\in B(b')$, we have 
$bc'=s$; so $c' \in S(b,s)$.  So  $\M(c) \cap S(b,s) = \M(c) \cap C$. 

Since $\M(c) \cap 
S(b,s)$ is a contact element for every $c\in C$, it now follows that $C$ 
is a hypersurface. \end{proof}

%\medskip 
 Recall that the $C$ of this Proposition was more completely denoted 
$B\vdash s$, and it deserves the name of ``hypersurface parallel to 
$B$ at distance $s$''. It inherits a transversal orientation from that of $B$.

We have therefore  a generalization of Huygens' Theorem,  stated  
in \cite{Arnold} p.\ 250. The surfaces $B\vdash s$ mentioned are the ``wave 
fronts'', or, the ``dilatations'' of $B$ (\cite{Lie2} p.\ 14-15). 
The Huygens Theorem stated in Section \ref{HTX} is the 
special case where $B=S(a,r)$.
\begin{thm} Given a transversally oriented hypersurface $B$ in $M$. 
Then for small enough 
$s$, we have another hypersurface $B\vdash s$, which is an envelope 
of the spheres $S(b,s)$ as $b$ ranges over $B$. We have $ 
B\vdash (s+t) = (B\vdash s)\vdash t$,  for $t$ and $s$ small enough.
\end{thm}
The last assertion follows from Proposition \ref{asssx}, together 
with the characterization of $P\vdash s$ in terms of 
$a\triangleright_{s}$ (Subsection  \ref{85x}).

\medskip

\begin{remark}  {\em Let us note that if $b\sim b'\in B$, then the two contact elements 
$P:=\M(b)\cap B$ and $P'=\M(b')\cap B$ are in united position: we 
say that two contact elements $P$ and $P'$ are in {\em united 
position} if they are neighbours in the manifold of contact elements, {\em and} if  $b\in P'$ and $b'\in P$; this notion plays a central 
role in the  work of S.\ Lie, \cite{Lie2} p.\ 39, or \cite{Lie1} p.\ 480.}\end{remark}

\medskip

The construction of rays given by ``vectors'' (pairs $a$ and $b$ of 
points) should be contrasted with the construction of the ``flow'' of 
contact elements $P$; this is in some sense the relationship between 
the Lagrangian and the Hamiltonian description of the process of 
propagation in geometrical optics, 
as in \cite{Arnold}. The present note began as an attempt to 
complete the essentially synthetic/metric account of this 
relationship, given in loc.cit. p.\ 250.

%\subsection{Orthogonality of rays and wave fronts}\label{orthRayWavex}

%\ldots Lagrange vs Hamilton

   %\subsection{Flows of hypersurfaces}
   
      \section{Models based on SDG}
\label{modelx}
We consider in the present Section models for 
 $\di$, $\#$, 
 and 
$<$, which are built from a (commutative) local ring $R$ with a strict total order 
$<$. So the set of invertible elements in $R$ fall in two 
disjoint classes $R_{>0}$ and $R_{<0}$, both stable under addition, 
 and with $R_{>0}$ stable under multiplication and containing $1$.  
We have $x>y$ if $(x-y) \in R_{>0}$, and $x-y$ if $(y-x)\in R_{<}$. For 
$x-y$ invertible, one  has the dichotomy: $x<y$ or $x>y$.
We write $x\# y$ for $y-x$ invertible. The {\em absolute value} 
function is the function $R_{<0}\cup R_{>0} \to R_{>0}$ given by $x 
\mapsto -x$ if $x<0$ and $x\mapsto x$ if $x>0$.

 We require in the present Section, that positive square roots of elements in $R_{>0}$ exist 
uniquely.

If $\vx =(x_{1},\ldots ,x_{n})\in 
R^{n}$ has at least one of the $x_{i}$s invertible, we say that $\vx$ 
is a {\em proper} vector. For a proper vector $\vx$, we ask that $\sum 
x_{i}^{2}$ is $>0$, so $\sqrt{\sum x_{i}^{2}} \in R_{>0}$ exists. So 
for a proper vector $\vx$
we may define $|\vx |:=\sqrt{\sum x_{i}^{2}}$, equivalently, using the canonical inner product $\langle -,- \rangle$,
$$|\vx|^2 = \langle \vx,\vx\rangle.$$

We say that vectors $\vx$ and $\vy$ in $R^{n}$ are {\em apart} (written $\vx \# 
\vy$) if $\vy - \vx$ is a proper vector; then $|\vy - \vx |\in 
R_{>0}$ defines a metric 
$\di (\vx , \vy )$,  
 the {\em distance} between $\vx$ and $\vy$.

We take for $\sim$ on $R^{n}$ the standard one from SDG, namely 
$$(x_{1}, \ldots ,x_{n})\sim (y_{1},\ldots , y_{n}) \mbox{  if  }
(x_{i}-y_{i})\cdot (x_{j}-y_{j}) =0 \mbox{ for all }i,j = 1, \ldots 
n.$$

For $\va\in R^n$ and $r>0$, the sphere $S(\va,r)$ is the set of $\vx\in R^n$ with $\langle \va-\vx, \va-\vx\rangle = r^2$. If $0\in S(\va.r)$, we thus have $\langle \va,\va\rangle  =r^2$. 
The monad $\M (0)$ is $D(n)$; elements $\vd$ in $D(n)$ satisfy $\langle \vd,\vd\rangle =0$ (but  $\langle \vd,\vd\rangle =0$ does not imply $\vd \in D(n)$ unless $n=1$).
Consider also the hyperplane $H=\va^{\perp}\subseteq R^n$ (where $\va$ is a proper vector). Then 
$$D(n) \cap S(\va,r) = D(n) \cap H;$$
for if $\vd\in  D(n) \cap S(\va,r)$, we have
$\langle \va-\vd,a-\vd \rangle =r^2$, and if we calculate the left hand side here, we get
$$\langle \va,\va\rangle - 2 \langle \vd,\va\rangle + \langle \vd,\vd\rangle = r^2 -2 \langle \vd,\va\rangle, $$
and this can only be $r^2$ if $\langle \vd,\va\rangle =0$, so $\vd\in \va^{\perp} =H$. Conversely, if $\vd\in D(n)\cap H$, the same calculation (essentially "Pythagoras") shows that $\vd\in S(\va,r)$.

Since the metric is invariant under translations, we therefore also have
\begin{prop}\label{fu1x}
For $A$ any sphere and $\vb\in A$ we have
$\M(\vb) \cap A = \M (\vb) \cap H$, where $\va$ is the center of $A$,  and  $H$ denotes the hyperplane orthogonal to $\vb-\va$ through $\vb$.\end{prop}

The following Proposition depends on $R$ being a model for the KL axiomatics. We shall use coordinate free notation, in particular, for an $n$-dimensional vector space $V$, we have a subset $D(V)\subseteq V$,  defined as the image of  $D(n)\subseteq R^n$	 under some linear isomorphism $R^n\to V$; it does not depend on the choice of such isomorphism.

\begin{prop}\label{fu2x}Let $H$ and $K$ be (affine) hyperplanes in an $n$-dimen\-sional vector space $V$, and assume $\vb\in H\cap K$. Then
$\M (\vb) \cap H \subseteq K$ implies $H=K$. \end{prop}
\begin{proof} Again by parallel translation, we may assume that $\vb=\underline{0}$, so the hyperplanes $H$ and $K$ are linear subspaces of $V$ of  dimension $n-1$. So for dimension reasons, it suffices to prove $H\subseteq K$. Let $\phi:  V\to R$ be a surjective linear map with kernel $K$. To prove $H\subseteq K$, we should prove that $\phi$ annihilates $H$. By assumption, $\phi$ annihilates $D(V) \cap H$. Since $H$ is a linear retract of $V$, $D(V)\cap H = D(H)$ (see the proof of Proposition 1.2.4 in \cite{SGM}). So the linear map $\phi \mid _{H}:H \to R$ restricts to the zero map on $D(H)$. By the KL axiom (see e.g.\ \cite{SGM}, I.3, the zero map is the only linear map which does so, so $\phi$ restricts to $0$ on $H$.\end{proof}

Consider two spheres $A$ and $C$, with centers are $\va$ and $\vc$, respectively, and with $\va\# \vc$. Let $\vb\in A\cap C$. 
\begin{prop}\label{fu3x}If $\M (\vb) \cap A \subseteq C$, then $\M (\vb )\cap A  = \M (\vb) \cap C$. 
\end{prop}
\begin{proof}Let $H$  be the hyperplane associated to $A,\vb$ as in Proposition \ref{fu1x}, and let $K$ similarly be the hyperplane associated to $C,\vb$.
So $$\M (\vb) \cap H   = M(b) \cap A \subseteq \M(\vb) \cap C = \M (\vb) \cap K.$$
 Then Proposition \ref{fu2x} gives that $H=K$, and therefore the middle equality sign in
$$\M(\vb) \cap A    =\M(\vb) \cap H   = \M(\vb) \cap K  = \M(\vb) \cap C.$$ 
\end{proof}

\begin{prop} \label{focusx} For any $\vx\in R^{n}$, the monad $\M (\vx)$ is focused, 
with $\vx$ as focus.
\end{prop} 
\begin{proof}  For simplicity of notation, we prove that the monad $\M 
(\underline{0})=D(n)$ is focused. Now 
 $D(n)$ may be described as $\{\vd =(d_{1}, \ldots 
,d_{n}) \in R^{n}
\mid d_{i}\cdot d_{j}=0 \mbox{  for all $i,j$  }\}$. For $D(1)$, one writes just 
$D$; so $d\in D$ means $d^{2}=0$. So assume that $\vx =(x_{1}, 
\ldots ,x_{n})\in D(n)$ has $\vx \sim  \vd$ for all $\vd \in D(n)$. This 
means that for all $i,j$, we have $0= (x_{i}-d_{i})\cdot (x_{j}-d_{j})$; but
$$(x_{i}-d_{i})\cdot (x_{j}-d_{j})= -x_{i}\cdot d_{j}-x_{j}\cdot 
d_{i},$$
using $x_{i}\cdot x_{j}=0$ and $d_{i}\cdot d_{j} =0$. Take in 
particular $\vd$ of the form $(d, 0, \ldots,0)$, with $d\in D$ and take $i=j=1$. Then the 
equation gives for all $d\in D$ that $0= -2 x_{1}\cdot d$, so for all $d\in D$, we have 
$x_{1}\cdot d =0$. By cancelling the universally quantified $d$, we 
get $x_{1}=0$, by the basic  axiom for SDG, see \cite{SDG} I.1. 
Similarly, we get $x_{2}=0$ etc., so $\vx = \underline{0}$.\end{proof}

More generally, one may prove that for suitable non-degenerate subspaces 
$H\subseteq R^{n}$, the set $\M (\vz) \cap H$ is focused (for $\vz\in H$). This 
applies e.g.\ to a hyperplane $H$ (zero set of a proper affine map 
$R^{n}\to R$); for, then $\M \vz \cap H$ is $equiv D(n-1)$. It also applies to spheres, which is our main 
concern.

A subset $N$ of a monad $\M (\vx)$, with $\vx\in N$, need not be 
focused. Consider for example some $\epsilon$ with 
$\epsilon^{2}=0$, and consider the set $ \{ \epsilon\cdot \vx \mid \vx \in 
R^{n}\}$. It is a subset of $\M (\underline{0})=D(n)$ and contains $\underline{0}$. But 
it is not focused: {\em any} pair of elements $\epsilon\cdot \vx$ and 
$\epsilon \cdot \vy$ in it are neighbours. In ``geometry based on the ring of 
dual numbers ${\mathbb R}[\epsilon]$" (as in Hjelmslev's 
\cite{Hjelmslev}), one may 
define a $\sim$-relation, based on elements $d$ of square 0, but the 
resulting monads will not be focused, essentially by the above 
argument. For, in \cite{Hjelmslev}, there is a fixed $\epsilon \in R$ such 
that every element $d\in R$ with $d^{2}=0$ is of the form $\epsilon 
\cdot x$.

We shall prove that in $R^n$, with metric derived from the standard inner product $\langle-,-\rangle$, the touching of spheres  is focused.
We first prove 

 \begin{prop}\label{87xx} Let $U\subseteq R^n$ be a finite dimensional  linear subspace 
 (meaning here that $U$ is a linear direct summand in $R^n$); 
 Then the following are equivalent for a vector $\va\in R^n$ (assumed to be proper, i.e.\ $\va \#0$)
 
 1) $\va \perp U$
 
 2) for all $\vd\in D(n)\cap U$, we have $|\va| = |\va+\vd|$.
 \end{prop}
 \begin{proof}Since $\va$ is proper, then so is $\va+\vd$, so the two norms mentioned are $>0$, so their equality is equivalent to the equality of their squares,
 i.e. to
 $$\langle \va,\va \rangle = \langle \va+\vd, \va+\vd \rangle = \langle \va\va \rangle + 2\langle \va,\vd \rangle .$$ So the Proposition says that 
 $\va\perp U$ iff $\langle \va, \vd \rangle  =0$ for all $\vd\in D(n)\cap U$; or, $\va\perp U$ iff  the linear functional $\langle \va, -\rangle :V\to R$ restricts to $0$ on $D(n)\cap U$. The latter assertion is equivalent (like in the proof of Proposition \ref{fu2x})  to saying that the functional restricts to $0$ on all of $U$, i.e.\ to $\va\perp U$.
 \end{proof}

  Translated into geometric terms with $\vb\in U$, where $U$ is an affine subspace of $R^n	$, this implies,  for $\va$ apart from $U$:
  \begin{prop}\label{footx} Let $\va\in R^n$ and $\va$ apart from $U$. If $\vb$ is the foot  (orthogonal projection) of $\va$ on $U$, then $\va$ has the same distance to all  points $\vb' \in U$ with $\vb' \sim \vb$; and conversely.\end{prop}

 Now consider the special case where the affine subspace $U$ is a hyperplane $H$, so of dimension $n-1$.
  We consider in the following spheres $A$, whose center $\va$ are apart from $H$, so $\di(\va,\vx)$ is defined for every $\vx\in H$.
 Spheres $A$ in $R^n$ have dimension $n-1$. Then one has (like in the Proposition \ref{fu3x})  that if
 $\vb\in A\cap H$ has $\M (\vb) \cap H \subseteq  A$ or  $\M (\vb)\cap A \subseteq H$, then we have equality $\M(\vb) \cap H= \M(\vb) \cap A$, i.e.\  $H$ and $A$ touch each other in $\vb$.
 Therefore
 \begin{prop}\label{touchxx} 1) Let $H$ be a hyperplane, and let a point $a$ (apart from $H$) have foot $b$ on $H$. Then $H$ touches the sphere $A:=S(a,r)$ in $b$, where $r=ab$; 2) Conversely, if a sphere with center $a$ touches $H$ in a point $b$, then $b$ is the foot of $a$ on $H$.\end{prop}
 \begin{proof} For 1): Since $b$ is the foot of $a$, we have $ab' =r$ for all $b'\in H$ with $b'\sim b$, which is to say
 $\M (b)\cap H \subseteq S(a,r)$, which as argued is the touching condition. -- For 2), the assumption gives that $\M (b) \cap H \subseteq A$, so all points $b'\in H$ with $b'\sim b$ have same distance to $a$, so $b$ is the foot of $a$ on $H$, by  Proposition \ref{footx} (the "conversely"-part).
 \end{proof}
 
 \begin{prop}\label{perx} If $A$ and $B$ are spheres with centres $a$ 
and $b$, respectively (with $a\# b$), then if $x$ and $y$ are in 
$A\cap B$, $\langle x-y, a-b \rangle =0$
\end{prop}
\begin{proof} 
Let $r$ and $s$ be the radii of the two spheres. 
Then we have  $\langle (x-a), (x-a)\rangle =r^{2}$ and 
$\langle (x-b), (x-b)\rangle =s^{2}$, and similarly for $y$. Then 
$\langle x-y, a-b \rangle =0$ follows by simple arithmetic. \end{proof}
 Since feet (orthogonal projections) are unique, it follows that a sphere can touch a hyperplane in at most one point; and furthermore, the touching set is focused, being of the form $\M (b)\cap H$ for a hyperplane in $R^n	$,  so is of the form $D(n-1)$. This proves the first assertion in
 \begin{cor}\label{AHx} In $R^n$, the touching of spheres with hyperplanes is focused. Also, the touching of two spheres (non-concentric in the sense that their centers are apart) is focused. 
 \end{cor}
 \begin{proof} To prove the second assertion, let the spheres be $A$ and $C$, with centers $\va$ and $\vc$, respectively (with $\va\#\vc$), and assume that $A$ and $C$ touch in a point $\vb$. Let $H$ be the hyperplane through $\vb$  and orthogonal to the line connecting $\va$ and $\vc$. Then $\vb$ is the foot of $\va$ on $H$.  So by Proposition \ref{touchxx}, $A$ touches $H$ in $\vb$. Similarly $C$ touches $H$ in $\vb$, so 
 $$\M(\vb) \cap A =   \M(\vb) \cap  H=\M(\vb) \cap C,$$
 and the middle set is known to be focused, since $H$ is a hyperplane i $R^n$. 
 We finally have to argue that the touching point $\vb$ is unique. 
 If $\vb_{1}$ were another point in which the spheres touch, the hyperplane $H_{1}$ through $\vb_{1}$ and orthogonal to the line from $\va$ to $\vc$ is the same as $H$ by Proposition \ref{perx}, so $\vb_{1}$, being the foot of $\va$ on $H_{1}=H$, is the same as $\vb$,
   proving the uniqueness of a possible touching point $b$. \end{proof}

 \begin{remark} \label{redherring} {\em If $A$ and $H$ are as in Proposition \ref{touchxx}, we have of course $\M(\vb)\cap H\subseteq A\cap H$, but we cannot conclude that  $\M(\vb)\cap H$ equals  $A\cap H$; to wit, the unit sphere $A$ in $R^3$ with center $(0,0,1)$ touches the $xy$-plane $H$ with touching set $D(2)\times \{0\}$, but $A\cap H = \{(x,y,0)\mid x^2 + y^2 =0\}$, which is in general larger than $D(2)\times \{0\}$. (In fact the set $\{(x,y)\in R^2 \mid x^2 +y^2 =0\}$ has been a puzzling "red herring" since the early days of SDG; in what sense is it an infinitesimal object?) So the touching of spheres $A$ and hyperplanes $H$ in $b$ need not be "clean" in the sense that the touching set $\M (b) \cap H$ equals $A\cap H$. In $R^2$, this "cleanness" can be asserted, see the preliminary version \cite{arx} (where it is incorrectly stated for general $R^n$).} \end{remark}

\bigskip

\small

\noindent Dept.\ of Mathematics, Aarhus University

\noindent Feb.\ 2017

\noindent kock (at) math.au.dk

\end{document}